\documentclass[a4paper,12pt,oneside]{article}

\usepackage[latin1]{inputenc}
\usepackage[english]{babel}
\usepackage{amsthm}
\usepackage{amssymb}
\usepackage{amsmath}
\usepackage{hyperref}
\usepackage{mathrsfs}
\usepackage{soul}
\usepackage{pstricks}
\usepackage[nottoc]{tocbibind}

 \usepackage{tikz}

\usepackage{graphicx}
\graphicspath{{images/}}

\usepackage{longtable,geometry}
\usepackage{color}

\usepackage{hyperref}

\newtheoremstyle{note}{12pt}{12pt}{}{}{\bfseries}{.}{.5em}{}
\title{\LARGE\textbf{Cherry flows with non trivial attractors}}
\author{Liviana Palmisano\\ Uppsala University\\
Uppsala, Sweden}
\newtheorem{theo}[equation]{Theorem}
\newtheorem{prop}[equation]{Proposition}

\newtheorem{defin}[equation]{Definition}

\numberwithin{equation}{section}
\newtheorem{lemma}[equation]{Lemma}

\newcommand{\N}{{\mathbb N}}
\newcommand{\Z}{{\mathbb Z}}
\newcommand{\R}{{\mathbb R}}

\newcommand{\T}{{\mathbb T}^2}
\renewcommand{\S}{{\mathbb S}^1}

\newcommand{\Cinf}{{{\mathcal C}^\infty}}

\newcommand{\C}{{\mathcal C}}

\usepackage{makeidx}

\providecommand{\norm}[1]{{\lVert #1 \rVert}_{{\mathcal C}^n}}

\usepackage{hyperref}

\begin{document}
\maketitle
\author
\textcolor{blue}{}\global\long\def\TDD#1{{\color{red}To\, Do(#1)}}

\begin{abstract}
We provide an example of Cherry flow (\emph{i.e.} $\Cinf$ flow on the $2$-dimensional torus with a sink and a saddle) having quasi-minimal set which is an attractor. 
The first return map for such a flow, constructed also in the paper, is a $\Cinf$ circle map having a flat interval and a non-trivial wandering interval. 
\end{abstract}
\section{Introduction}

\subsection{Discussion and statement of results}
In 1986 Poincar\'e conjectured the existence of a $\Cinf$ flow on the $2$-dimensional torus with a non-trivial minimal set. 
The conjecture was later disproved by Denjoy. In 1937 Cherry proved that the conjecture of Poincar\'e is true if you ask 
the existence of a non-trivial quasi-minimal set.
He constructed $\Cinf$ flows on the torus, called Cherry flows, without closed trajectories and with two singularities, a sink and a saddle, both hyperbolic. 
The quasi-minimal set in this case is the 
complement of the stable manifold of the sink. 
The geometrical properties of the quasi-minimal set have been studied in \cite{MvSdMM}, \cite{my3}, \cite{my2}. 
In all the known examples the quasi-minimal set has zero Lebesgue measure and almost every orbit converges to the sink.
In this paper we construct examples of Cherry flows having quasi-minimal 
sets which are non-trivial attractors. These examples have two coexisting attractors: the sink and the quasi-minimal set.

The construction is based on the study of the first return map for a Cherry flow and on one of the fundamental question in circle dynamics: whether a circle map is 
conjugate to a rotation. 
One of the first steps in this area was done by Denjoy \cite{Denjoy}. He proved that any $\C^1$ diffeomorphism with irrational rotation number and 
with derivative of bounded variation is conjugate to a rotation. 
In the same paper he showed that the hypothesis on the derivative is essential by giving examples of $\C^1$ diffeomorphisms with irrational rotation number 
which are not conjugate to a rotation.  
The reason for which the Denjoy function is not conjugate to a rotation is the presence of a wandering interval. 
A wandering interval $I$ has the property that, for all $n\in\Z$, $f^n(I)\cap I=\emptyset$ and $f^n_{|I}$ is a diffeomorphism. 
Since then examples of this kind, called Denjoy counterexamples, have attracted attention of many mathematicians.

Another important contribution given by Katok in \cite{Katok} is a Denjoy counterexample which is $\Cinf$ 
everywhere with the exception of one point which is a non-flat critical point of the function. 
A $\Cinf$ circle function being not conjugate to a rotation and with at most two flat critical points was constructed by Hall in \cite{hall}. 
In \cite{my} the author of this paper generalized Hall's construction showing a piece-wise $\Cinf$ Denjoy counterexample having a flat half-critical point. 

Our aim is to further extend the ideas of \cite{hall} and \cite{my}. We construct a $\Cinf$ Denjoy counterexample 
having an arc of critical points all being flat. Our result is the following:
 
 \begin{theo}\label{conden}
For any irrational number $\rho \in[0,1)$ there exists a $\Cinf$, non-decreasing circle map $f$ of degree one and an arc $U$ of the circle such that:

\begin{itemize}
	\item $f(U)$ is a point,
	\item $f$ has rotation number $\rho$, 
	\item $f$ is a Denjoy counterexample.
\end{itemize}
\end{theo}
Notice that the arc $U$ is not a wandering interval for $f$, the forward iterates are not diffeomorphisms. 
%
%

Our result not only contributes to the field of Denjoy counterexamples but also can be applied to study the quasi-minimal set
\footnote{The closure of any non-trivial recurrent trajectories is called a quasi-minimal set.} of Cherry flows. 

Using Theorem\ref{conden} we are now able to study more subtle topological properties of the quasi-minimal set.
 We recall that any Cherry flow has a well defined rotation number $\rho \in [0,1)$  equal to the rotation number of its first return map to a chosen
 Poincar\'e section and we state:

\begin{theo}\label{3}
For any irrational number $\rho\in[0,1)$ there exists a Cherry flow with rotation number $\rho$ and with quasi-minimal set which is an attractor. 
The basin of attraction of the quasi-minimal set has non-empty interior.  
\end{theo}
Thanks to this result we are now able to control the long-time behavior of the orbits of a large class of Cherry flows.

An interesting problem arising at this point is understanding physical measures of the flows constructed in Theorem \ref{3}. 
Observe that in the standard case of Cherry flows with a saddle point and an attractive point the answer to this question is easy: the Dirac
delta at the attractive point is the only physical measure. For the flows of Theorem \ref{3} all questions remain open. 
For example we expect that physical measures are non-trivial measures concentrated on the attractor. 
For the reference we mention that the only studies of non-trivial physical
measures concerned the case of inverted Cherry flows, viz. flows with a saddle point and a repulsive point, see \cite{SV}, \cite{my1}, \cite{Ya}.

\subsection{Notations and Definitions}
By $\S=\mathbb R/\mathbb Z$ we denote the unit circle and by $R_{\rho}:\S\mapsto\S $, $\rho\in\R$, the map defined as
\[
R_{\rho}(\theta+\mathbb Z) = (\theta+ \rho+\mathbb Z)
\]
and called rotation by $\rho$.

   Let $\pi:\R\mapsto\S$ be the projection of the real line to the circle and let $f:\S\mapsto\S$ be a continuous map. 
   We call a function $F:\R\mapsto\R$ a lift of $f$ if
\[
\pi\circ F=f\circ\pi. 
\]
A lift $F$ inherits regularity properties of $f$ e.g. continuity, differentiability, smoothness. In the following we refer to maps of $\S$ with minuscule letters e.g $f$.  The corresponding capital letter $F$ denotes a lift with the property $F(0)\in [0,1)$, which is unique.

We recall that a continuous function $f:\S\mapsto\S$ has degree one if for all $x\in\R$, $F(x+1)=F(x)+1$.
Moreover, $f$ is said non-decreasing if its lift is non-decreasing.

\subsection{Rotation number}\label{rotnum}
Let $f:\mathbb{S}^1\mapsto\mathbb{S}^1$ be a non-decreasing, continuous function of degree one then the limit
\begin{equation*}
\lim_{n\to\infty}\frac{F^{n}(x)}{n}
\end{equation*}
exists for any $x\in\R$ and it is independent of $x$. This limit is called the \emph{rotation number} of $f$ and will be denoted by $\rho(f)$.\\

$\rho$ as a function of $f$ is continuous:
\begin{prop}\label{famrot1}
If ${(f_n)}_{n \in \N}$ is a family of continuous, non-decreasing functions of degree one and if $f_n\to f_0$ uniformly then $\rho(f_n)\to \rho(f_0)$.  
\end{prop}

The rotation number is also non-decreasing:
\begin{prop}
If $F_1< F_2$ then $\rho(f_1)\leq\rho(f_2)$. Moreover if $\rho(f_1)$ or $\rho(f_2)$  is irrational then $\rho(f_1)<\rho(f_2)$.
\end{prop}

\section{Proof of Theorem \ref{conden}}
\subsection{Technical Lemmas}

\begin{lemma}\label{eq}
Let $f:\S\mapsto\S$  be a continuous, non-decreasing, degree one function with irrational rotation number $\rho\in\left[0,1\right)$. Then the following statements are equivalent:
\begin{enumerate}
\item $f$ is a Denjoy counterexample,
\item $f$ has not dense orbits,
\item $f$ has a wandering interval,\emph{ i.e.} there exists a non-empty interval $I\subset\S$ such that, for all $n,m\in\Z$, $n\neq m$, $f^n(I)\cap f^m(I)=\emptyset$,
\item there exists an interval $I\subset\S$ such that $\left|I\right|>0$ and $\left|f^n(I)\right|\rightarrow0$ for $n\rightarrow +\infty$.
 \end{enumerate}
\end{lemma}
The proof of this lemma can be found in \cite{hall}, p. 263.\\

In the following, for all $m\in\N$ the derivative of order $m$ of a real function $F$ in a point $x$ will be denoted by $F^{(m)}(x)$.

For $F:\R\mapsto\R$, a $j$-times differentiable map we define $$ ||F||_{\mathcal C^j}=\sup_{x\in\R \text{\\ }0\leq i\leq j}\left|F^{(i)}(x)\right|$$.

\begin{lemma}
	Let $n\in\N$ odd (resp. even) and let $ {f}  : \S\mapsto \S$ be a $\C^n$, non-decreasing, degree one function, 
	for which there exists an interval $U=\pi(a,b)$ such that for all $x\in (a,b)$ and for all $m\in\N$, $F^{(m)}(x)=0$.

	Then, $\forall \epsilon\in(0,\frac{1}{4})$, $\forall \delta \in (0,1)$, 
	there exists a $\C^n$, non-decreasing, degree one function,  
	$\tilde{f} = \tilde{f}_{n,\epsilon, \delta} : \S \mapsto \S$, satisfying the following conditions:

	\begin{enumerate}
		\item $\norm{\tilde{F} - F} < \delta$,
		\item $\rho(\tilde{f} ) = \rho(f)$,
		\item $|\tilde{F}^{(1)}(x) - F^{(1)}(x)| < \delta F^{(1)}(x)$, $\forall x \notin (a-\frac{1}{4}(b-a), b+\frac{1}{4}(b-a) )$,		
		\item for all $m\in\N$, $\tilde{F}^{(m)}(x) = 0 \Leftrightarrow x \in (a, a+\frac{\epsilon}{2^n})\cup (a+\frac{2\epsilon}{2^n},b)$ (resp. 
		$x \in (a, b-\frac{2\epsilon}{2^n})\cup (b-\frac{\epsilon}{2^n},b)$).
	\end{enumerate}
\label{lemma:primoleft}
\end{lemma}
\begin{proof}
The proof of this lemma is the same of the proof of Lemma 5 in \cite{hall} or Lemma 2.2 in \cite{my}. 
\end{proof}

%
%

\begin{lemma}\label{lemma:secondoleft}
	Let $n\in\N$ odd (resp. even) and let $\tilde{f}  : \S\mapsto \S$ be a $\C^n$, non-decreasing, degree one function.\\
	We assume that there exist $\epsilon>0$ and two intervals $I = (a, a+\frac{\epsilon}{2^n})$ 
	(resp. $I = (a, b-\frac{2\epsilon}{2^n})$) and $J= (a+\frac{2\epsilon}{2^n},b)$ (resp. $J= (b-\frac{\epsilon}{2^n},b)$ ), such that for all $m\in\N$:
	$$ x \in I\cup J \Leftrightarrow \tilde F^{(m)}(x) = 0 $$
	
	and suppose that there exists a positive integer $r$ such that $\tilde{f}^r(I) \in (b-\frac{\epsilon}{2^{n+1}},b)$ (resp. $\tilde{f}^r(J) \in (a,a+\frac{\epsilon}{2^{n+1}})$).
	\par
	Then, $\forall \sigma \in (0,1)$, $\exists g = g_{n,\epsilon,\sigma} : \S \mapsto \S$ a $\C^{n+1}$, non-decreasing, degree one function satisfying the following conditions:
	\begin{enumerate}
		\item $\norm{ G - \tilde{F}  } <  \sigma$,
		\item $\rho(g) = \rho(\tilde{f} )$,
		\item $| G^{(1)}(x) - \tilde{F}^{(1)}(x)|  <  \sigma \tilde{F}^{(1)}(x)$, $\forall x \notin (a-\frac{1}{4}(b-a), b+\frac{1}{4}(b-a) )$,
		\item $G^{(m)}(x) = 0 \Leftrightarrow x\in J$ (resp. $I$),
		\item on some left-sided (resp. right-sided) neighborhood of $a+\frac{2\epsilon}{2^n}$ (resp. $b-\frac{2\epsilon}{2^n}$), $f$ can be represented as 
		$h_{l,n}((a+\frac{2\epsilon}{2^n}-x )^{n+2})$  (resp. $h_{r,n}((x-b-\frac{2\epsilon}{2^n} )^{n+3})$) where
$h_{l,n}$ (resp. $h_{r,n}$ )is a $\Cinf$-diffeomorphism on an open neighborhood of $a+\frac{2\epsilon}{2^n}$ (resp. $b-\frac{2\epsilon}{2^n}$),
          \item $G^r(I)\Subset (b-\frac{\epsilon}{2^{n+1}},b)$ (resp. $G^r(J)\Subset (a, a+\frac{\epsilon}{2^{n+1}})$)
		
	\end{enumerate}
\end{lemma}
\begin{proof}
The proof of this lemma is the same of the proof of Lemma 6 in \cite{hall} or Lemma 2.3 in \cite{my}.
\end{proof}

\subsection{Proof of Theorem \ref{conden}}

%
\begin{proof}
Let $\rho\in\left[0,1\right)$ and $\epsilon\in\left(0,\frac{1}{4}\right)$ be two fixed irrational numbers. 
We start with a $\C^1$, non-decreasing, degree one, circle map $f$ and an interval $U_0=\pi(a_0,b_0)$ of length $l$ such that
\begin{equation}\label{eq:4}
 l>2\sum_{i=0}^{\infty}\frac{\epsilon}{2^i}
 \end{equation}
 and 
\begin{itemize}
 \item $\forall m\in\N$, $F^{(m)}(x)=0 \Leftrightarrow x\in (a_0,b_0)$, 
 \item on some right-sided neighborhood of $b_0$, $f$ can be represented as $h_{r,0}((x-b_0 )^2 )$ where
$h_{r,0}$ is a $\Cinf$-diffeomorphism on an open neighborhood of $b_0$,
\item on some left-sided neighborhood of $a_0$, $f$ can be represented as $h_{l,0}((a_0-x )^2 )$ where
$h_{l,0}$ is a $\Cinf$-diffeomorphism on an open neighborhood of $a_0$.
\end{itemize}

For all $t\in[0,1)$ we consider $F_t = F+ t$ and we choose $t_0$ such that $\rho(f_{t_0})=\rho$.
The existence of such a $t_0$ is guaranteed by Proposition \ref{famrot1}. 

We denote by $f_0=f_{t_0}$ and by $I$ a proper subset of $f_0^{-1}(b_0-\epsilon, b_0)$ 
having strictly positive length. 

The Denjoy counterexample will be constructed as limit of a sequence of functions:

\begin{displaymath}
(f_n:\S\mapsto\S)_{n \in \N},
\end{displaymath}
for which there exist a sequence of arcs $\{U_n\}_{n \in \N}$ with 
\[
U_n=\pi((a_n,b_n)) \text{ and } U_{n}\subset U_{n-1}
\]

and a sequence of integers $\{r_n\}_{n \in \N}$ fulfilling 
\begin{displaymath}
1=r_0 \text{ and } r_n<r_{n+1}
\end{displaymath}
 such that, for all $i\in\N$ the following conditions are satisfied:
\begin{enumerate}
\item $f_i$ is a $\C^i$, non-decreasing, degree one map,
\item $\rho(f_i)=\rho$,
\item $\forall m\in\N$, ${F^{(m)}_i}(x)=0$ if and only if $x\in (a_i,b_i)$,
\item if $i$ is odd (resp. even) on some right-sided neighborhood of $b_i$, $f$ can be represented as $h_{r,i}((x-b_i )^{i+3})$  (resp. $h_{r,i}((x-b_i )^{i+2})$) where
$h_{r,i}$ is a $\Cinf$-diffeomorphism on an open neighborhood of $b_i$,
\item if $i$ is odd (resp. even) on some left-sided neighborhood of $a_i$, $f$ can be represented as $h_{l,i}((a_i-x )^{i+1})$ (resp. $h_{l,i}((a_i-x )^{i+2})$) where
$h_{l,i}$ is a $\Cinf$-diffeomorphism on an open neighborhood of $a_i$,
\item $|U_i|= |U_{i-1}|-\frac{2\epsilon}{2^{i-1}}$,
\item $\left\|F_i-F_{i-1}\right\|_{\mathcal C^{i}}\leq\frac{1}{2^{i}}$,
\item $| F_i^{(1)}(x) - F_{i-1}^{(1)}(x) | < \frac{1}{2^{i+1}}F_i^{(1)}(x)$, $\forall x \notin (a_i-\frac{1}{4}(b_i-a_i), b_i+\frac{1}{4}(b_i-a_i) )$,

\item $0<\left|f_i^j\left(I\right)\right|<\frac{1}{2^{k-1}}$, if $r_{k-1}\leq j<r_k$ for $k \in \{1,2,\ldots,i\}$,
\item $f_i^j\left(I\right)\cap U_i=\emptyset$, if $0\leq j<r_i$ and 

\[
f_{i}^{r_i}\left(I\right)\Subset 
\bigg\{
\begin{tabular}{ccc}
  $\pi(b_i-\frac{\epsilon}{2^{i}},b_i)$  if $i$ is even  \\
 $ \pi(a_i, a_i+\frac{\epsilon}{2^{i}})$ if $i$ is odd 
  \end{tabular}
\]

\end{enumerate}
The sequence $\{f_n\}_{n \in \N}$ will be constructed iteratively. 
Suppose that we have $f_n$. We produce $f_{n+1}$ by perturbing $f_n$ in the way that the conditions $(1)$-$(10)$ are still satisfied. 
We suppose $n$ odd (the case of $n$ even is analogous), then $f_{n}^{r_n}\left(I\right)\Subset \pi(a_n, a_n+\frac{\epsilon}{2^{n}})$. 

We fix $\delta\in(0,1)$. By Lemma \ref{lemma:primoleft} applied to $f_n$ we produce a $\C^n$, non-decreasing, degree one function 
$\tilde{f}_{n,\delta}:=\tilde{f}_{n,\epsilon, \frac{\delta}{2^{n+2}}}$ such that $\forall m\in\N$, 
$$\tilde F^{(m)}_{n,\delta}(x)=0 \Leftrightarrow x\in (a_n, a_n+\frac{\epsilon}{2^n})\cup (a_n+\frac{2\epsilon}{2^n}, b_n).$$
 
\par
We study now the orbit of $I$ under $\tilde{f}_{n,\delta}$. 
Observe that by construction $\tilde{F}_{n,\delta}^i\rightarrow F_n^i$ for $\delta\rightarrow 0$ uniformly for $i\in\{1,2,\ldots,r_n\}$, 
then we can fix $\delta'<\delta<1$, such that:
\begin{displaymath}
\left|\tilde{f}_{n,\delta'}^j\left(I\right)\right|<\frac{1}{2^{k-1}}\textrm{ for all } r_{k-1}\leq j <r_k, \textrm{  } k\in\{1,2,\ldots,n\},
\end{displaymath}

\begin{equation}\label{paura1}
\tilde{f}_{n,\delta'}^j\left(I\right)\cap U_n=\emptyset\textrm{ for all }0\leq j<r_n
\end{equation}
and
\begin{equation}\label{paura}
\tilde{f}_{n,\delta'}^{r_n}\left(I\right)\subset \pi(a_n, a_n+\frac{\epsilon}{2^{n}}).
\end{equation}

Let us now study the orbits of $J_n=\pi(a_n, a_n+\frac{\epsilon}{2^{n}})$ under $\tilde{f}_{n,\delta'}$. We may have two different cases:

\begin{itemize}
\item there exists $m>0$ such that $\tilde{f}_{n,\delta'}^m(J_{n}) \in U_n$,
\item for all $m>0$, $\tilde{f}_{n,\delta'}^m(J_{n}) \notin U_n$.
\end{itemize}

Observe that the second situation never occurs. In fact, in such a case we can construct $g:\S\to\S$ being non-decreasing function of degree one, which is equal 
to $\tilde{f}_{n,\delta'}$ everywhere
except on $U_n\setminus \pi(a_n+\frac{\epsilon}{2^{n}}, b_n)$. Moreover, $g$ can be chosen so that on some right-sided neighborhood of $a_n+\frac{\epsilon}{2^{n}}$  
it is equal to $h_{r,n}((x-(a_n+\frac{\epsilon}{2^{n}}))^{n+1})$ for some $\Cinf$ -diffeomorphism $h_{r,n}$. 
Such $g$ belongs to the class of functions with a flat interval (being $J_n$ in
this case) studied in \cite{a} (see property (5)). We notice that $g^m (J_n) =\tilde{f}_{n,\delta'}^m(J_{n})$ thus the orbits of $g$ are not dense. 
By Lemma \ref{eq} it has a wandering interval and therefore contradicts Corollary to Theorem 1 in \cite{a}.

\par
We study now the first case. The aim is to prove that there exists $\delta''$ (smaller than $\delta'$ if necessary) such that 
\begin{equation}\label{eq:3}
\tilde{f}_{n,\delta''}^m(J_{n})\Subset \pi(b_n-\frac{\epsilon}{2^{n+1}},b_n).
\end{equation}
Because of the fact that the rotation number is irrational $f_n^m(U_n)$ does not enter into $U_n$ and we may suppose that $f_n^m(U_n)$ 
approaches $U_n$ from the right side. Observe that the functions $\tilde{f}_{n,\delta}$ approximate $f_n$, in particular
\[
	      \tilde{f}_{n,\delta}^m(J_{n}) \to f_n^m(U_n),
\]
as $\delta\searrow 0$. Moreover $\delta \mapsto \tilde{f}_{n,\delta}^m(J_{n})$ is continuous and the set
\[
	A=\{\tilde{f}_{n,\delta}^m(J_{n}):\delta \in [0,\delta']\},
\]
is an interval of $\S$ containing $\tilde{f}_{n,\delta'}^m(J_n)$ and $f_n^m(U_n)$. The proof is concluded once we observe that, 
because of the fact that the rotation number $\rho$ is irrational, $A \cap J_n = \emptyset$, so $A$ covers a portion of the interior of 
$\pi(a_n+\frac{\epsilon}{2^{n}},b_n)$. 
Because we are supposing that $f_n^m(U_n)$ approaches $U_n$ from the right side, we can choose $\delta''<\delta'$ if necessary such that 
$\tilde{f}_{n,\delta}^m(J_{n})$ is contained 
in a sub-interval of $\pi(b_n-\frac{\epsilon}{2^{n+1}},b_n)$. 

We denote by $a_{n+1}=a_n+\frac{2\epsilon}{2^{n}}$, $b_{n+1}=b_n$ and consequently  $U_{n+1}=\pi((a_n+\frac{2\epsilon}{2^{n}},b_n))$. 
We fix $\sigma > 0$ and we apply Lemma \ref{lemma:secondoleft} to the function $\tilde{f}_{n,\delta''}$. We get a $\C^{n+1}$, non-decreasing, 
degree one map $f_{n+1,\sigma}=f_{n+1,\epsilon,\frac{\sigma}{2^{n+1}}}:\S\mapsto\S$ such that for all $m\in\N$, 
$$F^{(m)}_{n+1,\sigma}(x)=0 \Leftrightarrow x\in (a_{n+1}, b_{n+1}).$$

Finally we define $r_{n+1}=r_n+m$. Since $f_{n+1,\sigma}^i\rightarrow \tilde{f}_{n}^i$ uniformly for all $i\in\{1,2,\ldots,r_n+m\}$ 
and since $\tilde{f}_{n}^i\left(I\right)$ is a singleton for $i>r_n$, then  we get that (we set $f_{n+1}=f_{n+1,\sigma}$):
\begin{displaymath}
\left|f_{n+1}^j\left(I\right)\right|<\frac{1}{2^{k-1}}
\end{displaymath}
for all $r_{k-1}\leq j<r_k$, with $k\in\{1,2,\ldots,n\}$, and
\begin{displaymath}
\left|f_{n+1}^j\left(I\right)\right|<\frac{1}{2^n}
\end{displaymath}
if $r_n\leq j< r_{n+1}$.\\
By (\ref{paura1}),  (\ref{paura}) and (\ref{eq:3}), since $f_{n+1,\sigma}^i\rightarrow \tilde{f}_n^i$ uniformly for all $i\in\{1,2,\ldots,r_{n+1}\}$, then
\begin{displaymath}
f_{n+1}^i\left(I\right)\cap U_{n+1}=\emptyset\textrm{ for all }0\leq i<r_{n+1}
\end{displaymath}
and

\begin{displaymath}
f_{n+1}^{r_{n+1}}\left(I\right) \Subset \pi(b_n-\frac{\epsilon}{2^{n+1}},b_n).
\end{displaymath}

So we have constructed a sequence of functions $(f_k)_{k \in \N}$ satisfying $(1)$-$(10)$ for all $n\geq 1$. 
By condition $(7)$ the sequence $(f_k)_{k \in \N}$ converges in the sense of the norm $\|\cdot\|_{\mathcal C^n}$ for all $n$. 
Then the limit function $f$  is a $\Cinf$, non-decreasing, degree one circle map which has rotation number $\rho$ (see $(1), (2), (3), (4)$ and $(5)$). 
Let consider now $a=\lim_n a_n$ and $b=\lim_n b_n$ and $U=\pi(a,b)$.  By $(3)$ and $(8)$ $F^{(1)}(x)=0$ if and only if $x\in (a,b)$ 
and by $(6)$ and the assumption on the size of $U_0$ (see \ref{eq:4}), $U$ has strictly positive length.
\par
Finally, by conditions $(9)$ and $(10)$,
\begin{displaymath}
\left|f_n^i\left(I\right)\right|\rightarrow 0\textrm{ for }i\rightarrow+\infty
\end{displaymath}
uniformly in $n$, and then
\begin{displaymath}
\left|f^i\left(I\right)\right|\rightarrow 0\textrm{ if }i\rightarrow+\infty.
\end{displaymath}
By Lemma \ref{eq} $f$ has a wandering interval and thus it is not conjugate to a rotation.

\end{proof}
\section{Applications: Cherry Flows}
\subsection{Basic Definitions}
Let $X$ be a $C^{\infty}$ vector field on the torus $T^2$. 
Denote the flow through a point $x$ by $t\rightarrow X_t(x)$.

The $\omega$-\textit{limit set} of a positive semi-trajectory $\gamma^{+}(x)$ is the set
\begin{displaymath}
\omega(\gamma^{+}(x))=\left\{y:\exists t_n\rightarrow\infty \textmd{ with } X_{t_n}(x)\rightarrow y \right\},
\end{displaymath}
and $\alpha$-\textit{limit set} of a negative  semi-trajectory $\gamma^{-}(x)$ is
\begin{displaymath}
\alpha(\gamma^{-}(x))=\left\{y:\exists t_n\rightarrow\infty \textmd{ with } X_{-t_n}(x)\rightarrow y \right\}.
\end{displaymath}
The $\omega$-limit set ($\alpha$-limit set) of any positive (negative) semi-trajectory of the trajectory $\gamma$ is called $\omega$-limit set $\omega(\gamma)$ of $\gamma$ ($\alpha$-limit set $\alpha(\gamma)$ of $\gamma$).

The trajectory is $\omega$-recurrent ($\alpha$-recurrent), if it is contained in its $\omega$-limit set ($\alpha$-limit set). The trajectory is \textit{recurrent} if it is both $\omega-$recurrent and $\alpha-$recurrent. A recurrent trajectory is \textit{non-trivial} if it is neither a fixed point nor a periodic trajectory.

\begin{defin}
A Cherry flow is a $\Cinf$ flow on the torus $\T$ without closed trajectories which has exactly two singularities, a sink and a saddle, both hyperbolic.
\end{defin}

The first example of such a flow was given by Cherry in \cite{Cherry}.

\subsection{Basic Properties of Cherry Flows}
We state now basic properties of Cherry flows. For more details the reader can refer to \cite{NZE}, \cite{MvSdMM}, \cite{mythesis}.

Given a Cherry flow $X$ we can always find a circle $\S$ which is not retractable to a point and is everywhere transverse to $X$. It is called a Poincar\'e section.

The first return map $f$ of $X$ to $\S$ turns out to be a continuous degree one circle map which is constant on an interval $U$ 
(it corresponds to the interval of points which are attracted to the singularities).

\begin{defin}\label{QM}
Let $\gamma$ be a non-trivial recurrent trajectory. Then the closure $\overline{\gamma}$ of $\gamma$ is called a quasi-minimal set.
\end{defin}

Every Cherry flow has only one quasi-minimal set which is locally homeomorphic to the Cartesian product of a Cantor set $\Omega$ and a segment $I$. 
Moreover $\Omega$ is equal to the non-wandering set\footnote{the set of the points $x$ such that for any open neighborhood $V\ni x$ there exists an integer $n>0$ such that the intersection of $V$ and $f^n(V)$ is non-empty.} of the first return function $f$.

We give now the definition of attractor. For more details on the concept of attractor the reader can refer to \cite{milnor}.
\begin{defin}\label{attractor}
Let $M$ be a smooth compact manifold.
A closed subset $A\subset M$ is called an attractor if it satisfies two conditions:
\begin{enumerate}
\item the realm of attraction $\Lambda(A)$ consisting of all points $x\in M$ for which $\omega(x)\subset A$, has strictly positive Lebesgue measure;
\item there is no strictly smaller closed set $A' \subset A$ so that $\Lambda (A')$ coincides with $\Lambda (A)$
up to a set of zero Lebesgue measure.
\end{enumerate}
The first condition says that there is some positive possibility that a randomly
chosen point will be attracted to $A$, and the second says that every part of $A$ plays
an essential role.
\end{defin}

We are now ready to construct an example of Cherry flow whose quasi-minimal set is a attractor.  
\subsection{Proof of Theorem \ref{3}}

\begin{proof}
Let $\rho$ be an irrational number and let $f$ be the Denjoy counterexample constructed in Theorem \ref{conden} having rotation number $\rho$. 
By suspending $f$, we get a Cherry flow $X$ having first return map $f$, see Appendix $1$. We claim that the quasi-minimal set $Q$ of $X$ is an attactor. 
Let $\S$ be the Poincar\'e section and let $I\subset\S$ be the wandering interval of $f$. 
Then $\omega(I)\subset Q$ and the realm of attraction $\Lambda(Q)$ has positive Lebesgue measure. Point 1 of Definition \ref{attractor} is satisfied.

For point 2, by Definition \ref{QM}, $Q=\overline{\gamma}$ with $\gamma$ a non-trivial recurrent orbit. In particular $\gamma\subset\omega(\gamma)\subset\overline\gamma$.
Let $Q' \subset Q$ be a closed set such that $\Lambda (Q')=\Lambda (Q)$. Then $\gamma\subset Q'$ and consequently $Q=Q'$. The proof is now complete.
\end{proof} 
\section*{Acknowledgments}
The author is very grateful to Prof. J. Graczyk, who introduced her to the topic of the paper, to Prof. A. Avila for helpful discussions and to Prof. E. Pujals for the argument in Appendix $1$. The author was supported by funds allocated to the implementation of the international co-funded project in the years 2014-2018, 3038/7.PR/2014/2, and by the EU grant PCOFUND-GA-2012-600415.

\section{Appendix 1}\label{App1}

Given a $C^\infty$ circle map with a flat interval one can construct a corresponding Cherry flow. If the circle map has a critical exponent $\ell>0$ this will lead to a Cherry flow with a hyperbolic saddle point. In the case when the behavior of the circle map is flat near the boundary of the flat interval one has to pay more attention to whether the flow is indeed $C^\infty$. In this appendix we will sketch a construction, suggested by E. Pujals, of a saddle-like singularity with any given flat transition map.

\begin{theo} Given a $C^\infty$ homeomorphism $f:[0,1]\to [0,1]$ flat at $0$ (i.e. all derivatives vanish at $0$) then there exists a $C^\infty$ vector field of the form
\begin{equation}\label{vector}
\frac{dx}{dt}=v(x), \text{ }  \frac{dy}{dt}=-w(x)
\end{equation}
where $v$ and $w$ are $C^\infty$ non-negative functions flat at $0$ such that the transition map from  $[0,1]\times\{1\}\to \{1\}\times [0,1]$ is $f$. That is the orbit starting at $(x,1)$ exits the unit square in $(1, f(x))$.
\end{theo}
\begin{proof}
Observe, one need the following condition on $v$ and $w$
$$
\int_x^1 \frac{dx}{v(x)}=\int_0^t dt=-\int_1^{f(x) }\frac{dy}{w(y)}.
$$
For $x=1$ this equation always holds. So it is suffcient that the equation obtained by taking the derivative is satisfied. This gives
\begin{equation}\label{w}
w(f(x))=v(x)\cdot Df(x).
\end{equation}
Let 
$$
v(x)=e^{-\frac{1}{Df(x)}}.
$$
Using (\ref{w}), we define the function $w$. The corresponding vector field  (\ref{vector}) will have $f$ as its transition map. Both function $v$ and $w$ are $C^\infty$ and flat at $0$. This can be seen as follows. A calculation shows that the derivatives of $v$ are finite sums of the form
$$
D^{n}v(x)=\sum_{\underline p=(p, p_1, p_2,\dots, p_{n+1})} A_{\underline p}\cdot v(x)\cdot \frac{1}{Df(x)^p}\cdot \prod_{k\le n}( D^{k}f(x))^{p_k}
$$
with $A_{\underline p}\in\Z$.
The factor $v(x)\cdot \frac{1}{Df(x)^p}$ will go to zero when $x\to 0$. This means that $v$ is $C^\infty$ and flat at $0$. The function $w$ has derivatives of a similar shape. Namely, this follows from (\ref{w}),
$$
D^{n}w(f(x))=\sum_{\underline p=(p, p_1, p_2,\dots, p_{n+1})}B_{\underline p}\cdot v(x)\cdot \frac{1}{Df(x)^p}\cdot \prod_{k\le n}( D^{k}f(x))^{p_k}
$$
with $B_{\underline p}\in\Z$.
This implies again that $w$ is $C^\infty$ and flat at zero.
\end{proof}

\end{document}